\documentclass[12pt]{amsart}

\usepackage{amsmath}
\usepackage{amssymb}
\usepackage{ascmac}
\usepackage{amsthm}

\makeatletter

\@addtoreset{equation}{section}
\makeatother 

\theoremstyle{plain}
\newtheorem{thm}{Theorem}[section]
\newtheorem*{thm*}{Theorem}
\newtheorem{prop}[thm]{Proposition}
\newtheorem{lem}[thm]{Lemma}

\theoremstyle{definition}

\theoremstyle{remark}
\newtheorem{rem}{Remark}[section]

\newcommand{\vol}{\operatorname{vol}}

\newcommand{\diam}{\operatorname{Diam}}

\usepackage{latexsym}

\title[A universal inequality for Neumann eigenvalues]{A universal inequality for Neumann eigenvalues of the Laplacian on a convex domain in Euclidean space}

\author{Kei Funano}
\address{Division of Mathematics \& Research Center for Pure and Applied Mathematics, Graduate School of Information Sciences, Tohoku University, 6-3-09 Aramaki-Aza-Aoba, Aoba-ku, Sendai 980-8579, Japan}
\email{kfunano@tohoku.ac.jp}

\date{\today}

\begin{document}
\maketitle

\begin{abstract}
We obtain a new upper bound for Neumann eigenvalues of the Laplacian on a bounded convex domain in Euclidean space. As an application of the upper bound we derive universal inequalities for Neumann eigenvalues of the Laplacian.
\end{abstract}

\section{Introduction}

The purpose of this article is to give a new upper bound for Neumann eigenvalues of the Laplacian on a bounded convex domain in Euclidean space and a universal inequality for Neumann eigenvalues of the Laplacian. 

Let $\Omega$ be a bounded domain in Euclidean space with piecewise smooth boundary. We denote by $\lambda_k(\Omega)$ the $k$th positive Neumann eigenvalues of the Laplacian on $\Omega$. For a finite sequence
$\{A_{\alpha}\}_{\alpha=0}^k$ of Borel subsets of $\Omega$ we set
\begin{align*}
\mathcal{D}(\{A_{\alpha}\}):=\min_{\alpha \neq \beta}
d(A_{\alpha},A_{\beta}),
 \end{align*}where $d(A_{\alpha},A_{\beta}):=\inf \{d(x,y) \mid x\in
 A_{\alpha},y\in A_{\beta}\}$ and $d$ is the Euclidean distance function. 

 Throughout this paper, we write $\alpha\lesssim \beta$ if $\alpha\leq c\beta$ for some universal concrete constant $c>0$ (which means $c$ does not depend on any parameters such as dimension and $k$ etc).

One of the main theorems in this paper is as follows.

\begin{thm}\label{Keythm} Let $\Omega$ be a bounded convex domain in $\mathbb{R}^n$ with piecewise smooth boundary and $\{A_{\alpha}\}_{\alpha=0}^k$ be a sequence of Borel subsets of $\Omega$. Then we have
\begin{align}\label{upest}
    \lambda_k(\Omega)\lesssim \frac{n^2}{(\mathcal{D}(\{A_{\alpha}\})\log (k+1))^2}\max_{\alpha=0,\cdots,k} \Big(\log \frac{\vol(\Omega)}{\vol(A_{\alpha})}\Big)^2.
\end{align}
\end{thm}

\begin{rem}The above theorem also holds for Neumann eigenvalues of the Laplacian on bounded convex domains in a manifold of nonnegative Ricci curvature. The proof only uses Lemma \ref{keylem} which follows from the Bishop-Gromov inequality.  
\end{rem}

In \cite{CGY1,CGY2} Chung-Grigory'an-Yau obtained 
\begin{align*}
    \lambda_k(\Omega)\lesssim \frac{1}{\mathcal{D}(\{A_{\alpha}\})^2}\max_{\alpha=0,\cdots,k} \Big(\log \frac{\vol(\Omega)}{\vol(A_{\alpha})}\Big)^2.
\end{align*}for a bounded (not necessarily convex) domain $\Omega$ and its Borel subsets $\{A_{\alpha}\}$ (See also \cite{FS,GH}). Compared to their inequality the inequality (\ref{upest}) is better for large $k$ if we fix $n$. Their inequality is better for large $n$ if we fix $k$.
Theorem \ref{Keythm} also gives an answer to Question 5.1 in \cite{F1} up to $n^2$ factor.

As an application of Theorem \ref{Keythm} we obtain the following universal inequality for Neumann eigenvalues of the Laplacian.
\begin{thm}\label{Mthm} Let $\Omega$ be a bounded convex domain in $\mathbb{R}^n$ with piecewise smooth boundary. Then we have
    \begin{align}\label{univ}
        \lambda_{k+1}(\Omega)\lesssim n^4 \lambda_k(\Omega).
    \end{align}
\end{thm}

Related with (\ref{univ}) the author conjectured in \cite{F0,F2} that 
\begin{align*}
    \lambda_{k+1}(\Omega)\lesssim  \lambda_k(\Omega) 
\end{align*}holds under the same assumption of Theorem \ref{Mthm}. In \cite[(1.3)]{F1} the author proved that 
\begin{align*}
    \lambda_{k+1}(\Omega)\lesssim (n\log k)^2\lambda_k(\Omega)
\end{align*}for a bounded convex domain $\Omega$. The inequality (\ref{univ}) avoids the dependence of $k$ for the upper bound of the ratios $\lambda_{k+1}(\Omega)/\lambda_k(\Omega)$ and gives a better inequality if $\log k \geq n$.
In \cite{F0,F2} the author proved a dimension-free universal inequality $\lambda_k(\Omega)\lesssim c^k \lambda_1(\Omega)$ for a bounded convex domain in $\mathbb{R}^n$ and for some universal constant $c>1$. In \cite[Theorem 1.5]{Liu} Liu showed an optimal universal inequality $\lambda_k(\Omega)\lesssim k^2\lambda_1(\Omega)$ under the same assumption. Thus $n^2$ factor is not needed for small $k$ (e.g., $k=2,3$) in (\ref{univ}). 
As mentioned in \cite[(1.5)]{F1} combining E.~Milman's result \cite{Mi1} with Cheng-Li's result \cite{CL} one can obtain $\lambda_k(\Omega)\gtrsim k^{2/n}\lambda_1(\Omega)$ under the same assumption. Together with Liu's inequality this shows
\begin{align*}
    \lambda_{k+1}(\Omega)\lesssim k^{2-2/n}\lambda_k(\Omega).
\end{align*}The inequality (\ref{univ}) is better than this inequality for large $k$ if we fix $n$. This inequality is better for large $n$ if we fix $k$.

\section{Preliminaries}

We collect several results to use in the proof of our theorems.
\begin{prop}[{\cite[8.2.1 Theorem]{B}}]\label{lower}Let $\Omega$ be a bounded domain in a Euclidean space with piecewise smooth
    boundary and $\{\Omega_{\alpha}\}_{\alpha=0}^{l}$ be a finite partition of $\Omega$ by subdomains in the sense that $\vol (\Omega_\alpha \cap \Omega_{\beta})=0$ for each different $\alpha,\beta$. Then we have
\begin{align*}
    \lambda_{l+1}(\Omega)\geq \min_{\alpha}\lambda_1(\Omega_{\alpha}).
\end{align*}
 \end{prop}

 Refer to \cite[Appendix $C_{+}$]{G} for a weak form of the above proposition.

 \begin{thm}[{\cite[$(1.2)$]{PW}}]\label{PW est}Let $\Omega$ be a bounded convex domain in a Euclidean space with piecewise smooth boundary. Then we have 
 \begin{align*}
     \lambda_1(\Omega)\geq \frac{\pi^2}{\diam (\Omega)^2}.
 \end{align*}
 \end{thm}

Combining Proposition \ref{lower} with Theorem \ref{PW est} in order to give a `good' lower bound for Neumann eigenvalues of the Laplacian it is enough to provide a `good' finite convex partition of the domain.

For an upper bound of Neumann eigenvalues we mention the following theorem.
 \begin{thm}[{\cite[Theorem 1.1]{Kr}}]\label{Cheng}Let $\Omega$ be a bounded convex domain in $\mathbb{R}^n$ with piecewise smooth boundary. For any natural number $k$ we have
     \begin{align*}
         \lambda_k(\Omega)\lesssim \frac{n^2k^2}{\diam (\Omega)^2}.
     \end{align*}
 \end{thm}

In order to construct a `good' partition we recall a Voronoi partition of a metric space. Let $X$ be a metric space and $\{x_{\alpha}\}_{\alpha\in I}$ be a subset of
  $X$. For each $\alpha \in I$ we define the \emph{Voronoi cell} $C_{\alpha}$
  associated with the point $x_{\alpha}$ as
  \begin{align*}
   C_{\alpha}:= \{x\in X \mid d(x,x_{\alpha})\leq d(x,x_{\beta}) \text{ for all
   }\beta\neq \alpha   \}.
\end{align*}If $X$ is a bounded convex domain $\Omega$ in a Euclidean space then $\{C_{\alpha}\}_{\alpha\in I}$ is a convex partition of
   $\Omega$ (the boundaries $\partial C_{\alpha}$ may overlap each
   other). Observe also that if the balls $\{ B(x_{\alpha},r)\}_{\alpha\in I}$ of
   radius $r$ covers $\Omega$ then $C_{\alpha} \subseteq B(x_{\alpha},r)$, and thus
   $\diam (C_{\alpha} )\leq 2r$ for any $\alpha\in I$.

\section{Proof of Theorems \ref{Keythm} and \ref{Mthm}}
We use the following key lemma to prove Theorem \ref{Keythm}.
\begin{lem}[{\cite[Lemma 3.1]{F3}}]\label{keylem}Let $\Omega$ be a bounded convex domain in $\mathbb{R}^n$ with a piecewise smooth boundary. Given $r>0$ suppose that $\{x_{\alpha}\}_{\alpha=0}^{l}$ is $r$-separated points in $\Omega$, i.e., $d(x_{\alpha},x_{\beta})\geq r$ for distinct $\alpha$, $\beta$. Then we have
\begin{align*}
    r\lesssim \frac{n}{\sqrt{\lambda_l(\Omega)}}.
\end{align*}
\end{lem}
\begin{proof}[Proof of Theorem \ref{Keythm}]Suppose that there is a sequence $\{ A_{\alpha}\}_{\alpha=0}^{k}$ of Borel subsets such that 
\begin{align*}
     \lambda_k(\Omega)\geq \frac{c n^2}{(\mathcal{D}(\{A_{\alpha}\})\log (k+1))^2}\max_{\alpha=0,\cdots,k} \Big(\log \frac{\vol (\Omega)}{\vol(A_{\alpha})}\Big)^2
\end{align*}for sufficiently large $c>0$. Since $(k+1)\vol(A_{\alpha})\leq \vol(\Omega)$ for some $\alpha$ we have
\begin{align}\label{key}
    \mathcal{D}(\{A_{\alpha}\})\geq \frac{c n}{\sqrt{\lambda_k(\Omega)}}=:r_0.
\end{align}For each $\alpha$ we fix a point $x_{\alpha}\in A_{\alpha}$. The sequence $\{x_{\alpha}\}_{\alpha=0}^{k}$ is then $r_0$-separated in $\Omega$ by (\ref{key}). 
By virtue of Lemma \ref{keylem} we get
\begin{align*}
   \frac{c n}{\sqrt{\lambda_k(\Omega)}}= r_0 \lesssim \frac{n}{\sqrt{\lambda_k(\Omega)}}.
\end{align*}For sufficiently large $c$ this is a contradiction. This completes the proof of the theorem.

\end{proof}
We can reduce the number of $\{A_{\alpha}\}$ in Theorem \ref{Keythm} as follows. 
\begin{lem}\label{lelemma} Let $\Omega$ be a convex domain in $\mathbb{R}^n$ and $\{A_{\alpha}\}_{\alpha=0}^{k-1}$ be a sequence of Borel subsets of $\Omega$. Then we have
\begin{align*}
    \lambda_k(\Omega)\lesssim \frac{n^2}{(\mathcal{D}(\{A_{\alpha}\})\log (k+1))^2}\max_{\alpha=0,\cdots,k-1} \Big(\log \frac{\vol(\Omega)}{\vol(A_{\alpha})}\Big)^2.
\end{align*}
\end{lem}The above lemma follows from Theorem \ref{Keythm} and \cite[Theorem 3.4]{F2}.

To prove Theorem \ref{Mthm} let us recall the Bishop-Gromov inequality in Riemannian geometry. See \cite[Lemma 3.4]{F3} for the proof in the case of convex domains in $\mathbb{R}^n$.
\begin{lem}[{Bishop-Gromov inequality}]Let $\Omega$ be a convex domain
 in $\mathbb{R}^n$. Then for any $x\in \Omega$ and any $R>r>0$ we have
 \begin{align*}
 \frac{ \vol (B(x,r)\cap \Omega)}{\vol (B(x,R)\cap \Omega)}\geq \Big(\frac{r}{R}\Big)^n.
  \end{align*}
  \end{lem}

  In the proof of Theorem \ref{Mthm} we make use of a similar argument as in \cite[Theorem 1.3]{F1}. 
\begin{proof}[Proof of Theorem \ref{Mthm}]
    Let
 $R:=cn^2/\sqrt{\lambda_{k+1}(\Omega)}$ where $c$ is a positive number
 specified later. Suppose that $\Omega$ includes $k+1$ $R$-separated net $\{x_{\alpha}\}_{\alpha=0}^{k}$ in
 $\Omega$. By Theorem \ref{Cheng} we have $\diam (\Omega) \leq c'n(k+1)/\sqrt{\lambda_{k+1}(\Omega)}$ for some universal constant $c'>0$. Applying the Bishop-Gromov inequality we have
 \begin{align*}
     \frac{\vol(B(x_{\alpha},R)\cap \Omega)}{\vol(\Omega)}\geq \frac{R^n}{(\diam \Omega)^n}\geq \Big(\frac{c}{c'(k+1)}\Big)^n\geq \frac{1}{(k+1)^n}
 \end{align*}for $c>c'$. By Lemma \ref{lelemma} we obtain 
 \begin{align*}
     \lambda_{k+1}(\Omega)\lesssim \frac{n^2(\log (k+1)^n)^2}{(\mathcal{D}(\{B(x_{\alpha},R)\cap \Omega\})\log (k+2))^2}\lesssim \frac{n^4}{R^2}=\frac{1}{c}\lambda_{k+1}(\Omega).
 \end{align*}For sufficiently large $c$ this is a
 contradiction. 
 
  Let $x_0, x_1,x_2,\cdots,x_l$ be maximal $R$-separated points in
   $\Omega$, where $l\leq k-1$. By the
 maximality we have $\Omega\subseteq
   \bigcup_{\alpha=0}^{l} B(x_{\alpha},R)$. If $\{
   \Omega_{\alpha} \}_{\alpha=0}^{l}$ is the Voronoi partition of $\Omega$ associated with
   $\{x_{\alpha}\}$ then we have $\diam (\Omega_{\alpha})\leq 2R$. Theorem \ref{PW est} thus yields $\lambda_1(\Omega_{\alpha})\gtrsim 1/R^2$ for each $\alpha$. According to Proposition \ref{lower} we obtain
   \begin{align*}
       \lambda_k(\Omega)\geq \min_{\alpha}\lambda_1(\Omega_{\alpha})\gtrsim \frac{1}{R^2}\gtrsim \frac{\lambda_{k+1}(\Omega)}{n^4}.
   \end{align*}This completes the proof of the theorem.
\end{proof}






{\it Acknowledgments.}
The author would like to express many thanks to the anonymous referees for their helpful and useful comments.  



\begin{thebibliography}{9999}
\bibitem{B}P.~Buser, \textit{Geometry and spectra of compact Riemann surfaces}. Reprint of the 1992 edition. Modern Birkh\"{a}user Classics. Birkh\"{a}user Boston, Inc., Boston, MA, 2010.
\bibitem{C}S.~Y.~Cheng, \textit{Eigenvalue comparison theorems and its geometric applications}, Math. Z. 143 (1975), no. 3, 289--297.
\bibitem{CL}S.~Y.~Cheng and P.~Li, \textit{Heat kernel estimates and lower bound of eigenvalues}. Comment.Math. Helv. 56 (1981), no. 3, 327--338.
\bibitem{CGY1}F.~R.~K.~Chung, A.~Grigor'yan and S.-T.~Yau, \textit{Upper bounds for eigenvalues of the discrete and continuous Laplace operators}, Adv. Math. 117 (1996), no. 2, 165--178.
\bibitem{CGY2}F.~R.~K.~Chung, A.~Grigor'yan and S.-T.~Yau,
        \textit{Eigenvalues and diameters for manifolds and graphs},
        Tsing Hua lectures on geometry \& analysis (Hsinchu,
        1990--1991), 79--105, Int. Press, Cambridge, MA, 1997.
\bibitem{F0}   K. Funano, \textit{Eigenvalues of Laplacian and multi-way isoperimetric constants on weighted Riemannian manifolds}, preprint arXiv:1307.3919.
\bibitem{F3}   K.~Funano, \textit{A note on domain monotonicity for the Neumann eigenvalues of the Laplacian}, to appear in Illinois J. Math.
\bibitem{F1}   K. Funano, \textit{Applications of the `ham sandwich theorem' to eigenvalues of the Laplacian}, Anal. Geom. Metr. Spaces 4 (2016), 317--325.
\bibitem{F2}   K. Funano, \textit{Estimates of eigenvalues of the Laplacian by a reduced number of subsets}, Israel J. Math. 217 (2017), no. 1, 413--433.
\bibitem{FS}K.~Funano and Y.~Sakurai, \textit{Upper bounds for
        higher-order Poincar\'e constants}. Trans. Amer. Math. Soc. {\bf
        373} (2020), no. 6, 4415--4436.
\bibitem{G}M.~Gromov, \textit{Metric structures for Riemannian and non-Riemannian spaces}, Based on the 1981 French original.
                      With appendices by M. Katz, P. Pansu and S. Semmes. Translated from the French by Sean Michael Bates. Reprint of the 2001 English edition.
                      Modern Birkh\"auser Classics. Birkh\"auser Boston, Inc., Boston, MA, 2007.
\bibitem{GH} N. Gozlan and R. Herry, \textit{Multiple sets exponential concentration and higher order eigenvalues}, preprint arXiv:1804.06133, to appear in Potential Anal..
\bibitem{GM} M. Gromov and V. D. Milman, \textit{A topological application of the isoperimetric inequality}, Amer. J. Math. 105 (1983), no. 4, 843--854.
\bibitem{Kr}P.~Kr\"{o}ger, \textit{On upper bounds for high order Neumann eigenvalues of convex domains in Euclidean
space}. (English summary) Proc. Amer. Math. Soc. 127 (1999), no. 6, 1665--1669.

 
\bibitem{Liu}  S. Liu, \textit{An optimal dimension-free upper bound
         for eigenvalue ratios}, preprint arXiv:1405.2213v3.
\bibitem{Mi1} E. Milman, \textit{On the role of convexity in isoperimetry, spectral gap and concentration}, Invent. Math. 177 (2009), no. 1, 1--43.
\bibitem{PW}L.~E.~Payne and H.~F.~Weinberger. \textit{An optimal poincar\'e inequality for convex domains}. Arch.
Rational Mech. Anal., 5:286–-292, 1960.
\end{thebibliography}
\end{document}